\documentclass[reqno,12pt]{amsart}
\usepackage{amscd,amsfonts,mathrsfs,amsthm,enumerate}
\usepackage{amssymb, amsmath}
\usepackage{stmaryrd}
\usepackage{epsfig}
\usepackage[alphabetic, sorted-cites]{amsrefs}

\thanks{The first author is supported by the grant of $\text{E}\Sigma\Pi\text{A}:$ PE1-417.}

\def\real     #1{{\mathbb R^{#1}}}
\def\complex  #1{{\mathbb C^{#1}}}

\def\dt       {\partial_{t}}

\def\equationcolor {\color{black}}
\def\textcolor     {\color{black}}

\def\bcoleq    {\begin{equation}\equationcolor}
\def\ecoleq    {\textcolor\end{equation}}
\def\bcoleqn   {\equationcolor\begin{eqnarray}}
\def\ecoleqn   {\end{eqnarray}\textcolor}

\def\gm{{\operatorname{g}_M}}
\def\gn{{\operatorname{g}_N}}

\def\gk{{\operatorname{g}_{M\times N}}}
\def\rm{{\operatorname{R}_M}}
\def\rn{{\operatorname{R}_N}}
\def\rk{{\operatorname{R}_{M\times N}}}
\def\sk{{\operatorname{s}_{M\times N}}}
\def\rind{\operatorname{R}}

\def\sind{\operatorname{s}}

\def\pr{\operatorname{pr}}

\def\dF{\operatorname{d}\hspace{-3pt}F}
\def\df{\operatorname{d}\hspace{-3pt}f}

\def\gind{\operatorname{g}}

\def\sym{\operatorname{Sym}}

\DeclareMathOperator*{\Ric}{Ric}

\DeclareMathOperator*{\trace}{trace}
\DeclareMathOperator*{\rank}{rank}

\newtheorem{theorem}{Theorem}[section]

\newtheorem*{thm}{Theorem}

\newtheorem{lemma}[theorem]{Lemma}

\theoremstyle{definition}
\newtheorem{remark}[theorem]{Remark}

\newcommand{\bfig}{\begin{figure}}
\newcommand{\efig}{\end{figure}}

\makeatletter
\def\pproof#1{\@ifnextchar[\opargproof
{\opargproof[\it Proof of #1.]}}
\def\opargproof[#1]{\par\noindent {\bf #1 }}

\makeatother

\numberwithin{equation}{section}

\begin{document}

\title[Mean Curvature Flow]{Evolution of contractions by mean curvature flow}
\author[Andreas Savas-Halilaj]{\textsc{Andreas Savas-Halilaj}}
\author[Knut Smoczyk]{\textsc{Knut Smoczyk}}
\address{Andreas Savas-Halilaj\newline
Institut f\"ur Differentialgeometrie\newline
Leibniz Universit\"at Hannover\newline
Welfengarten 1\newline
30167 Hannover, Germany\newline
{\sl E-mail address:} {\bf savasha@math.uni-hannover.de}
}
\address{Knut Smoczyk\newline
Institut f\"ur Differentialgeometrie\newline
and Riemann Center for Geometry and Physics\newline
Leibniz Universit\"at Hannover\newline
Welfengarten 1\newline
30167 Hannover, Germany\newline
{\sl E-mail address:} {\bf smoczyk@math.uni-hannover.de}
}

\date{}

\begin{abstract}
We investigate length decreasing maps $f:M\to N$ between Riemannian manifolds $M$, $N$
of dimensions $m\ge 2$ and $n$, respectively. Assuming that $M$ is compact and $N$ is complete such
that
$$\sec_M>-\sigma\quad\text{and}\quad{\Ric}_M\ge(m-1)\sigma\ge(m-1)\sec_N\ge-\mu,$$
where $\sigma$, $\mu$ are positive constants, we show that the mean curvature flow provides a smooth
homotopy of $f$ into a constant map.
\end{abstract}

\maketitle
%\setcounter{tocdepth}{1}
%\tableofcontents
%%%%%%%%%%%%%%%%%%%%%%%%%%%%%%%%%%%%%%%%%%%%%%%%%%%%%%%%%%%%%%%%%%%%%%%%%

\section{Introduction}

Let $f:M\to N$ be a smooth map between Riemannian manifolds. To any such $f$ we assign a geometric quantity called
$k$-{\it dilation}, which  measures how much the map stretches $k$-dimensional volumes. For example the $1$-dilation
coincides with the Lipschitz constant of the map. The map $f$ is called a {\it contraction} or
\textit{weakly length decreasing} if its $1$-dilation is less or equal to $1$. Equivalently,  the map $f$ is a contraction if
$f^{\ast}\gn\le\gm$, where $\gm,\gn$ stand for the Riemannian metrics of $M$ and $N$, respectively.
 In particular, the map $f$ will be called \textit{strictly length decreasing}
if $f^{\ast}\gn<\gm$ everywhere and an \textit{isometry} if $f^{\ast}\gn\equiv\gm$.

If $M=\mathbb{S}^m$ and $N=\mathbb{S}^n$ are unit spheres and $f:\mathbb{S}^m\to\mathbb{S}^n$ is
a strictly length decreasing map, then the diameter of $f(\mathbb{S}^m)$ is strictly less than $\pi$ which implies that
the map $f$ is not surjective. Hence, $f$ must be null-homotopic. Tsui and Wang \cite{tsui} proved that maps $f:\mathbb{S}^m\to\mathbb{S}^n$ between unit spheres
with $2$-dilation strictly less than $1$, or equivalently {\it strictly area decreasing}, are also homotopic to a
constant map. As it was shown by Guth \cite{guth,guth1} this result cannot be extended in the case of $k$-dilation
for $k\ge 3$.

Based on ideas developed in \cite{wang,tsui}, Lee and Lee \cite{lee} proved that any
strictly area decreasing map between compact Riemannian manifolds $M$ and $N$  whose sectional curvatures are bounded by $\sec_{M}\ge\sigma\ge\sec_{N}$
for some positive number $\sigma>0$, is homotopic by mean curvature flow to a constant map. We would like to point out
here that
the curvature assumptions can be relaxed even much further as it was shown in \cite{savas1}.
The goal of this short paper is to show that in the length decreasing case one can drop the
compactness assumption on $N$. More precisely we prove:

\begin{thm}\label{thmA}
Let $M$ and $N$ be two Riemannian manifolds with $M$ being compact and $N$ complete. Assume that $m=\dim M\ge 2$
and that there exist positive constants $\sigma$, $\mu$ such that the sectional curvatures $\sec_{M}$
of $M$ and $\sec_{N}$ of $N$ and the Ricci curvature ${\Ric}_M$ of $M$ satisfy
\begin{equation*}\label{curvcond3}
\sec_{M}>-\sigma\quad\text{and}\quad{\Ric}_M\ge(m-1)\sigma\ge(m-1)\sec_{N}\ge-\mu.
\end{equation*}
Let $f:M\to N$ be a strictly length decreasing map. Then the mean curvature flow of the graph
of $f$ remains the graph of a strictly length decreasing map, exists for all time and $f$ converges to a
constant map.
\end{thm}

In the case where $N$ is compact the above result is contained in our previous paper \cite{savas1}. The key argument to
remove the compactness is an estimate on the mean curvature vector field of the evolving graphs.
In particular, we prove that the norm of the mean curvature vector field remains uniformly bounded in time.
This will imply that the evolving graphs stay in compact regions of $M\times N$
on time intervals $[0,T)$, with $T<\infty$. Using this estimate, the blow-up analysis of Wang \cite{wang}
and White's regularity theorem \cite{white} we are able to prove that the maximal time $T$ of existence of the
flow is $\infty$. To prove the mean curvature estimate, we introduce a tensor on the normal bundle of the
evolving graphs and compare the maximum of the norm of the mean curvature with the biggest eigenvalue of
this tensor.

\section{Graphs}
\subsection{Basic facts}
We follow here the notations of our previous two papers \cite{savas1,savas}. The product manifold $M\times N$
will always be regarded as a Riemannian manifold equipped with the metric
$$\gk=\langle\cdot\,,\cdot\rangle:=\gm\times \gn.$$
The \textit{graph} of a map $f:M\to N$ is defined to be the submanifold
$$\Gamma(f):=\{(x,f(x))\in M\times N:x\in M\}$$
of $M\times N$. The graph $\Gamma(f)$ can be parametrized via the embedding $F:M\to M\times N$,
$F:=I_{M}\times f$, where $I_{M}$ is the identity map of $M$.

The Riemannian metric induced by $F$ on $M$ will be denoted by
$$\gind:=F^*\gk.$$
The two natural projections $\pi_{M}:M\times N\to M$ and $\pi_{N}:M\times N\to N$
are submersions, that is they are smooth and have maximal rank. The tangent bundle
of the product manifold $M\times N$, splits as a direct sum
\begin{equation*}
T(M\times N)=TM\oplus TN.
\end{equation*}
The four metric tensors $\gm,\gn,\gk$ and $\gind$ are related by
\begin{eqnarray*}
\gk&=&\pi_M^*\gm+\pi_N^*\gn\,,\label{met1}\\
\gind&=&F^*\gk=\gm+f^*\gn\,.\label{met2}
\end{eqnarray*}
As in \cite{savas1,savas}, define the symmetric $2$-tensors
\begin{eqnarray*}
\sk&:=&\pi_M^*\gm-\pi_{N}^*\gn\,,\label{met3}\\
\sind&:=&F^*\sk=\gm-f^*\gn\,.\label{met4}
\end{eqnarray*}
The Levi-Civita connection $\nabla^{\gk}$ associated to $\gk$ is
related to the Levi-Civita connections $\nabla^{\gm}$ on $(M,\gm)$ and $\nabla^{\gn}$ on
$(N,\gn)$ by
$$\nabla^{\gk}=\pi_M^*\nabla^{\gm}\oplus\pi_N^*\nabla^{\gn}\,.$$
The corresponding curvature operator $\rk$ on $M\times N$ with respect to the metric
$\gk$ is related to the curvature
operators $\rm$ on $(M,\gm)$ and $\rn$ on $(N,\gn)$ by
\begin{equation*}
\rk=\pi^{*}_{M}\rm\oplus\pi^{*}_{N}\rn.
\end{equation*}
The Levi-Civita connection on $M$ with respect to the induced metric
$\gind$ is denoted by $\nabla$, the curvature tensor by $\rind$ and the Ricci curvature by $\Ric$.

\subsection{The second fundamental form}
The differential $\dF$ of $F$ can be regarded as a section in the induced bundle  $F^*T(M\times N)\otimes T^*M$.
In the sequel we will denote all full connections on bundles over $M$ that are induced by
the Levi-Civita connection on $M\times N$ via the immersion $F:M\to M\times N$ by the same letter $\nabla$.
The covariant derivative of
$\dF$ is called the \textit{second fundamental form} of the immersion $F$
and it will be denoted by $A$. That is
$$A(v,w):=(\nabla\hspace{-2pt}\dF)(v,w),$$
for any vector fields $v,w\in TM$. If $\xi$ is a normal vector of the
graph, then the symmetric tensor $A_{\xi}$ given by
$$A_{\xi}(v,w):=\langle A(v,w),\xi\rangle$$
is called the \textit{second fundamental form with respect to the direction $\xi$}.

The trace of $A$ with respect to the metric $\gind$ is called the \textit{mean curvature
vector field} of $\Gamma(f)$ and it will be denoted by
$$H:={\trace}_{\gind}A.$$
Note that $H$ is a section in the normal bundle $\mathcal{N}M$. The graph $\Gamma(f)$ is
called \textit{minimal} if $H$ vanishes identically.

Every vector $V$ of $F^{\ast}T(M\times N)$ can be decomposed as
$$V=V^{\top}+V^{\perp},$$
where $V^{\top}$ stands for the {\it tangential component} and $V^{\perp}$ for the {\it normal
component} of $V$ along $F$. Introduce now the natural projection map
$\pr:F^*T(M\times N)\to\mathcal{N}M$, $\pr(V):=V^{\perp}$. We can express this map locally as
$$\pr(V)=V-\sum_{k,l=1}^m\gind^{kl}\langle V,\dF(\partial_k)\rangle\dF(\partial_l),$$
where $\{\partial_1,\dots,\partial_m\}$ is the basis of a local coordinate chart defined on an open
neighborhood of the manifold $M$ and $\gind^{kl}$ are the components of the inverse matrix $(\gind_{kl})^{-1}$,
where $\gind_{kl}=\gind(\partial_k,\partial_l)$, $1\le k,l \le m$.
The connection of the normal bundle will be denoted by the letter $\nabla^{\perp}$ and is defined by
$$\nabla^{\perp}_v\xi:=\pr\big(\nabla_v\xi\big),$$
where here $v\in TM$ and $\xi\in\mathcal{N}M$. The Laplacian with respect to $\nabla^{\perp}$ will be denoted by
$\Delta^{\perp}$.

By \textit{Gau\ss' equation} the tensors $\rind$ and $\rk$ are related by the formula
\begin{eqnarray*}
&&\big(\rind-F^*\rk\big)(v_1,w_1,v_2,w_2)\\
&&\quad\,\,\,\,=\big\langle A(v_1,v_2),A(w_1,w_2)\big\rangle-\big\langle A(v_1,w_2),A(w_1,v_2)\big\rangle,\nonumber\label{gauss}
\end{eqnarray*}
and the second fundamental form satisfies
the
\textit{Codazzi equation}
\begin{eqnarray*}
&&(\nabla_uA)(v,w)-(\nabla_vA)(u,w) \\
&&\quad\quad\quad\quad\quad\,\,\,\,=\rk\bigl(\dF(u),\dF(v),\dF(w)\bigr)-\dF\bigl(\rind(u,v,w)\bigr)\nonumber,\label{codazzi}
\end{eqnarray*}
for any $u,v,w, v_1,v_2,w_1,w_2\in TM$.

\subsection{Singular decomposition}\label{frames}
As in \cite{savas1,savas}, fix a point $x\in M$ and let
$\lambda^2_{1}\le\cdots\le\lambda^2_{m}$
be the eigenvalues at $x$ of $f^{*}\gn$ with respect to $\gm$. The corresponding values $\lambda_i\ge 0$,
$i\in\{1,\dots,m\}$, are called
\textit{singular values} of the differential $\df$ of $f$ at the point $x$. It turns out that the singular values
depend continuously on $x$. Set
$r:=\rank\df(x).$
Obviously, $r\le\min\{m,n\}$ and
$$\lambda_{1}=\cdots=\lambda_{m-r}=0.$$
At the point $x$ consider an orthonormal basis
$$\{\alpha_{1},\dots,\alpha_{m-r};\alpha_{m-r+1},
\dots,\alpha_{m}\}$$
with respect to $\gm$ which diagonalizes $f^*\gn$. Furthermore, at the point
$f(x)$ consider an orthonormal basis
$$\{\beta_{1},\dots,\beta_{n-r};\beta_{n-r+1},\dots,\beta_{n}\}$$
with respect to $\gn$ such that
$$\df(\alpha_{i})=\lambda_{i}\beta_{n-m+i},$$
for any $i\in\{m-r+1,\dots,m\}$.
Then one may define a special basis for the tangent and the normal space of the graph
in terms of the singular values. The vectors
\begin{equation*}
e_{i}:=\left\{
\begin{array}{ll}
\alpha _{i}&  , 1\le i\le m-r,\\
&  \\
\frac{1}{\sqrt{1+\lambda _{i}^{2}}}\left( \alpha
_{i}\oplus \lambda _{i}\beta _{n-m+i}\right)  &, m-r+1\leq
i\leq m,
\end{array}
\right.\label{tangent}
\end{equation*}
form an orthonormal basis with respect to the metric $\gk$ of the tangent space
$\dF\hspace{-2pt}\left(T_{x}M\right)$ of the graph $\Gamma(f)$ at
$x$.

The vectors
\begin{equation*}
\xi_{i}:=\left\{
\begin{array}{ll}
\beta _{i} & , 1\leq i\leq n-r,\\
&  \\
\frac{1}{\sqrt{1+\lambda _{i+m-n}^{2}}}\left( -\lambda
_{i+m-n}\alpha _{i+m-n}\oplus \beta _{i}\right) &, n-r+1\leq i\leq n, \\
\end{array}
\right.\label{normal}
\end{equation*}
give an orthonormal basis with respect to  $\gk$ of the normal space $\mathcal{N}_{x}M$ of the
graph $\Gamma(f)$ at the point $F(x)$. Note that
\begin{equation}\label{sind}
\sk(e_{i},e_{j})=\frac{1-\lambda^{2}_{i}}{1+\lambda^{2}_{i}}\delta_{ij},\quad 1\le i,j\le m.
\end{equation}
Consequently, the map $f$ is strictly length decreasing if and only if the symmetric $2$-tensor $\sind$ is positive.

Denote by $\sind^{\perp}$ the restriction of $\sk$ to the normal bundle of the graph. Then, we can readily check
that
\begin{eqnarray}
\hspace{-.5cm}
\sind^{\perp}(\xi_{i},\xi_{j})
&=&\begin{cases}
\displaystyle
-\delta_{ij}&\,, 1\le i\le n-r,\\[4pt]\displaystyle
-\frac{1-\lambda^{2}_{i+m-n}}{1+\lambda^{2}_{i+m-n}}\delta_{ij}&\,, n-r+1\le i\le n.
\end{cases}\label{normal2}
\end{eqnarray}
Hence, if there exists a positive constant $\varepsilon$ such that $\sind\ge\varepsilon\gind$, then
$$\sind^{\perp}\le-\varepsilon\gind^{\perp},$$
where $\gind^{\perp}$ stands for the restriction of
$\sk$ on $\mathcal{N}M$.
Furthermore,
\begin{equation}
\sk(e_{m-r+i},\xi_{n-r+j})=-\frac{2\lambda_{m-r+i}}{1+\lambda^{2}_{m-r+i}}\delta_{ij},
\quad 1\le i,j\le r.\label{mixed}
\end{equation}
Moreover, the value of $\sk$ on any other mixed term is zero.

\section{Evolution equations}
Let $M$ and $N$ be Riemannian manifolds, $f:M\to N$ a smooth map and $F:M\to M\times N$,
$F:=I_M\times f$, the parametrization of the graph $\Gamma(f)$ of $f$.
Deform the submanifold $\Gamma(f)$ by mean curvature flow in the product Riemannian manifold $M\times N$.
By this process we get a family of immersions
$F_t:M\to M\times N$, $t\in [0,T)$,
with initial condition $F_0=F$, where $0<T\le\infty$ shall denote the maximal time of existence.
From the compactness of $M$ it follows that the evolving submanifold stays a graph on an interval $[0,T_g)$
with $T_g\le T$, that is there exists a family of diffeomorphisms $\phi_{t}:M\to M$ and a family of maps
$f_t:M\to N$ such that
$F_t\circ\phi_t=I_M\times f_t,$
for any $t\in [0,T_g)$. In the matter of fact, under the assumptions of the Theorem, the singular values of $f$ remain
uniformly bounded in time and the solution of the mean curvature flow stays a graph as long as the flow exists.
This result follows from the next lemma, which still holds in the case where $N$ is complete.

\begin{lemma}[\cite{savas1}]\label{length}
Let $M$ and $N$ be Riemannian manifolds with $M$ being compact and $N$ complete. Assume that $m=\dim M\ge 2$
and that there exists a positive constant $\sigma$ such that the sectional curvatures $\sec_{M}$
of $M$ and $\sec_{N}$ of $N$ and the Ricci curvature ${\Ric}_M$ of $M$ satisfy
\begin{equation*}\label{curvcond3}
\sec_{M}>-\sigma\quad\text{and}\quad{\Ric}_M\ge(m-1)\sigma\ge(m-1)\sec_{N}.
\end{equation*}
Let $f:M\to N$ be a strictly length decreasing map such that $\sind\ge\varepsilon\gind$, where $\varepsilon$
is a positive constant. Then the inequality $\sind\ge\varepsilon\gind$ is preserved under the mean curvature flow. Furthermore,
$T_g=T$.
\end{lemma}

Now we claim that the norm of the mean curvature vector remains bounded in time. Inspired on ideas developed for
the Lagrangian mean curvature flow in \cite{smoczyk} (see also \cite{chau1} for the Lagrangian mean curvature flow
of non-compact euclidean domains in $\complex{m}$) we will compare the eigenvalues of the symmetric tensor
$H\otimes H$ with the biggest eigenvalue of
$-\sind^{\perp}$.

\begin{lemma}\label{projection}
Let $\xi$ be a local vector field along the graph of $f_{t_0}$ which is normal to $\Gamma(f_{t_0})$ at
a point $x_0$. The time derivative of $\pr$ at $\xi\in\mathcal{N}_{x_0}M$, when it is
regarded as a bundle map $\pr:F^{\ast}T(M\times N)\to F^{\ast}T(M\times N)$,
is given by
\begin{equation*}
(\nabla_{\dt}{\pr})(\xi)=-\sum_{j=1}^m\big\langle\xi,\nabla_{e_j}H\big\rangle e_j
=-\sum_{j=1}^m\big\langle\xi,\nabla^{\perp}_{e_j}H\big\rangle e_j,
\end{equation*}
where $\{e_1,\dots,e_m\}$ a local orthonormal frame in the tangent bundle of the graph.
Moreover, the time derivative of the natural projection at $\xi\in\mathcal{N}_{x_0}M$ ,
when it is
regarded as a map $\pr:F^{\ast}T(M\times N)\to\mathcal{N}M$, is zero. That is, $(\nabla^{\perp}_{\dt}\pr)(\xi)=0$.
\end{lemma}

\begin{proof}
Consider a local coordinate system $(x_1,\dots,x_m)$ around $x_0$ and suppose that the vectors
$\{\partial_1|_{x_0},\dots,\partial_m|_{x_0}\}$ are orthonormal. Extend them  now via parallel
transport to a frame field $\{\varepsilon_1,\dots,\varepsilon_m\}$ which is orthonormal
 with respect to the Riemannian metric $\gind(t_0)$. In order to simplify the notation we set
$e_i=\dF_{t_0}(\varepsilon_i)$, $1\le i\le m$. Extend also the vector $\xi$ arbitrarily.

Differentiating along the time direction, we get that
\begin{eqnarray*}
\left(\nabla_{\dt}\pr\right)(\xi)&=&\nabla_{\dt}\xi-\pr(\nabla_{\dt}\xi)\\
&&-\sum_{k,l=1}^m\gind^{kl}\langle\nabla_{\dt}\xi,\dF(\partial_k)\rangle\dF(\partial_l)\\
&&-\sum_{k,l=1}^m\dt\big(\gind^{kl}\big)\langle \xi,\dF(\partial_k)\rangle\dF(\partial_l)\\
&&-\sum_{k,l=1}^m\gind^{kl}\langle \xi,\nabla_{\dt}\dF(\partial_k)\rangle\dF(\partial_l)\\
&&-\sum_{k,l=1}^m\gind^{kl}\langle \xi,\dF(\partial_k)\rangle\nabla_{\dt}\dF(\partial_l).
\end{eqnarray*}
Because,
$$\dt\big(\gind^{kl}\big)=2\sum_{s,z=1}^m\gind^{ks}\gind^{zl}A_H(\partial_s,\partial_z)$$
we deduce that
\begin{eqnarray*}
\left(\nabla_{\dt}\pr\right)(v)&=&-\sum_{k,l=1}^m\gind^{kl}\langle \xi,\nabla_{\partial_k}H\rangle\dF(\partial_l)\\
&&-\sum_{k,l=1}^m\gind^{kl}\langle \xi,\dF(\partial_k)\rangle\nabla_{\partial_l}H \nonumber\\
&&-2\sum_{k,s,z,l=1}^m\gind^{ks}\gind^{zl}A_H(\partial_s,\partial_z)\langle \xi,\dF(\partial_k)\rangle\dF(\partial_l).
\end{eqnarray*}
Since, $\gind_{kl}(x_0,t_0)=\delta_{kl},$
we get that at this point it holds
\begin{eqnarray*}
\left(\nabla_{\dt}\pr\right)(\xi)&=&-\sum_{j=1}^m\langle \xi,\nabla_{e_j}H\rangle e_j
-\sum_{j=1}^m\langle \xi, e_j\rangle\nabla_{e_j}H \nonumber\\
&&-2\sum_{i,j=1}^mA_H(e_i,e_j)\langle \xi, e_i\rangle e_j\\
&=&-\sum_{j=1}^m\langle \xi,\nabla_{e_j}H\rangle e_j.
\end{eqnarray*}
Now, since $\pr\circ\pr=\pr$, we have that
$$\big(\nabla_{\dt}^{\perp}\pr\big)(\xi)=\pr\big(\nabla_{\dt}\pr(\xi)\big)-\pr\big(\nabla_{\dt}\xi\big)
=\pr\big\{\big(\nabla_{\dt}\pr\big)(\xi)\big\}=0.$$
This completes the proof of the lemma.
\end{proof}

In the next lemma we compute the evolution equation of $\sind^{\perp}$.
For that reason, it is necessary to extend $\sind^{\perp}$ on $F^{\ast}T(M\times N)$, by defining
$$\sind^{\perp}(V,W)=\sk\big(\pr(V),\pr(W)\big)$$
for any $V,W\in F^{\ast}T(M\times N)$.

\begin{lemma}
Let $\xi$ be a unit vector normal to the evolving submanifold at a fixed point $(x_0,t_0)$ in space time.
Then
\begin{eqnarray*}
\big(\nabla^{\perp}_{\dt}\sind^{\perp}-\Delta^{\perp}\sind^{\perp}\big)(\xi,\xi)
&=&2\sum_{i,j=1}^mA_{\xi}(e_i,e_j)\,\sk\big(A(e_i,e_j),\xi\big) \\
&-&2\sum_{i,j=1}^m\rk(e_i,e_j,e_i,\xi)\,\sk(e_j,\xi)\\
&-&2\sum_{i,j,k=1}^mA_{\xi}(e_i,e_j)A_{\xi}(e_i,e_k)\sind(e_j,e_k),
\end{eqnarray*}
for any orthonormal basis $\{e_1,\dots,e_m\}$ of $\dF_{t_0}(T_{x_0}M)$.
\end{lemma}

\begin{proof}
Let us compute at first the time derivative of $\sind^{\perp}$. Extend $\xi$ locally to a smooth vector
field along the graph. Then, using the fact that $\sk$ and $\pr$ are parallel tensors, we get that
\begin{eqnarray*}
(\nabla^{\perp}_{\dt}\sind^{\perp})(\pr(\xi),\pr(\xi))&=&\dt\big\{\sk(\pr(\xi),\pr(\xi))\big\}\\
&&-2\sk\big(\nabla^{\perp}_{\dt}\pr(\xi),\pr(\xi)\big)\\
&=&2\sk\big(\nabla_{\dt}\pr(\xi)-\nabla^{\perp}_{\dt}\pr(\xi),\xi\big)\\
&=&2\sk\big(\nabla_{\dt}\pr(\xi)-\pr(\nabla_{\dt}\xi),\xi\big)\\
&=&2\sk\big((\nabla_{\dt}\pr)(\xi),\xi\big).
\end{eqnarray*}
By virtue of Lemma \ref{projection}, we deduce that at the point $(x_0,t_0)$, the time derivative of
$\sind^{\perp}$ is given by
$$(\nabla^{\perp}_{\dt}\sind^{\perp})(\xi,\xi)=-2\sum_{j=1}^m\langle\nabla^{\perp}_{e_j}H,\xi\rangle e_j.$$
In the next step we compute the Laplacian of $\sind^{\perp}$. As usual, consider two vectors $\xi$ and
$\eta$ on $\mathcal{N}M$ and extend them locally to smooth normal vector fields. At first let us compute the
covariant derivative of $\sind^{\perp}$ with respect to the direction $e_i$. Using the fact that $\sk$ is parallel,
we have
\begin{eqnarray*}
\big(\nabla^{\perp}_{e_i}\sind^{\perp}\big)(\xi,\eta)&=&e_i\big\{\sk(\xi,\eta)\}-\sk(\nabla^{\perp}_{e_i}\xi,\eta)-
\sk(\xi,\nabla^{\perp}_{e_i}\eta)\\
&=&\sk\big(\nabla_{e_i}\xi-\nabla^{\perp}_{e_i}\xi,\eta\big)+\sk\big(\xi,\nabla_{e_i}\eta-\nabla^{\perp}_{e_i}\eta\big).
\end{eqnarray*}
Recall from the Weingarten formulas that
$$\nabla_{e_i}\xi=-\sum_{j=1}^mA_{\xi}(e_i,e_j)e_j+\nabla^{\perp}_{e_i}\xi.$$
Hence,
$$\big(\nabla^{\perp}_{e_i}\sind^{\perp}\big)(\xi,\eta)=-A_{\xi}(e_i,e_j)\sk(e_j,\eta)-A_{\eta}(e_i,e_j)\sk(e_j,\xi).$$
Differentiating once more in the direction of $e_i$, we get
\begin{eqnarray*}
\big(\nabla^{\perp}_{e_i}\nabla^{\perp}_{e_i}\sind^{\perp}\big)(\xi,\xi)
&=&-2\sum_{j=1}^m\langle(\nabla^{\perp}_{e_i}A)(e_j,e_i),\xi\rangle \sk(e_j,\xi)\\
&&-2\sum_{j=1}^m A_{\xi}(e_i,e_j)\sk(A(e_i,e_j),\xi)\\
&&+2\sum_{j,k=1}^mA_{\xi}(e_i,e_j)A_{\xi}(e_i,e_k)\sind(e_j,e_k).
\end{eqnarray*}
From the Codazzi equation we get
$$\big(\nabla^{\perp}_{e_i}A\big)(e_j,e_i)=\big(\nabla^{\perp}_{e_j}A\big)(e_i,e_i)+\pr\big({\rk}(e_i,e_j,e_i)\big).$$
Substituting the above relation in the formula of the Hessian of $\sind$ and then taking a trace, we see that
\begin{eqnarray*}
\big(\Delta^{\perp}\sind^{\perp}\big)(\xi,\xi)
&=&-2\sum_{j=1}^m\langle(\nabla^{\perp}_{e_j}H,\xi\rangle \sk(e_j,\xi)\\
&&-2\sum_{i,j=1}^m A_{\xi}(e_i,e_j)\sk(A(e_i,e_j),\xi)\\
&&+2\sum_{i,j,k=1}^mA_{\xi}(e_i,e_j)A_{\xi}(e_i,e_k)\sind(e_j,e_k)\\
&&+2\sum_{i,j=1}^m\rk(e_i,e_j,e_i,\xi)\sk(e_j,\xi).
\end{eqnarray*}
Combining the above formula for the Laplacian with the formula for the time derivative, we deduce the evolution
equation for $\sind^{\perp}$.
\end{proof}

Consider the symmetric tensor $\vartheta\in\sym\big(F^{\ast}T(M\times N)\otimes F^{\ast}T(M\times N)\big)$, given by
$$\vartheta(V,W):=H_{\pr(V)}\cdot H_{\pr(V)},$$
where $H_{\xi}=\trace A_{\xi}$ is the component of the mean curvature vector field in the direction of the
normal vector $\xi$.

\begin{lemma}
The symmetric tensor $\vartheta$ evolves in time under the mean curvature flow according to the formula
\begin{eqnarray*}
\big(\nabla^{\perp}_{\dt}\vartheta-\Delta^{\perp}\vartheta\big)(\xi,\xi)
&=&2\sum_{i,j=1}^mA_{H}(e_i,e_j)A_{\xi}(e_i,e_j)H_{\xi}
-2\sum_{i=1}^m\langle\nabla^{\perp}_{e_i}H,\xi\rangle^{2} \\
&&-2\sum_{i=1}^{m}\rk\big(H,e_i,e_i,\xi\big)H_{\xi}
\end{eqnarray*}
for any normal vector $\xi$ in the normal bundle of the submanifold.
\end{lemma}

\begin{proof}
At first let us compute the time derivative of $\vartheta$. Fix a point $(x_0,t_0)$ in space-time and consider
a unit normal vector $\xi$ of $\Gamma(f_{t_0})$ at the point $x_0$. Now extend $\xi$ to a local smooth vector
field.

Computing and then estimating at $(x_0,t_0)$, we get that
\begin{eqnarray*}
\big(\nabla^{\perp}_{\partial_t}\vartheta\big)(\xi,\xi)
&=&\dt\big\{\vartheta(\pr(\xi),\pr(\xi))\big\}-2\vartheta\big(\nabla^{\perp}_{\dt}\pr(\xi),\pr(\xi)\big)\\
&=&2\langle \nabla^{\perp}_{\dt}H,\pr(\xi)\rangle H_{\pr(\xi)}
+2\langle H, \nabla^{\perp}_{\dt}\pr(\xi)\rangle H_{\pr(\xi)}\\
&&-2\langle H, \nabla^{\perp}_{\dt}\pr(\xi)\rangle H_{\pr(\xi)}\\
&=&2\big\langle\nabla^{\perp}_{\dt}H,\xi\big\rangle H_{\xi}.
\end{eqnarray*}
From the evolution equation of the mean curvature vector $H$ (see \cite[Corollary 3.8]{smoczyk1}) we deduce that at
the point $x_0$ it holds
\begin{equation*}
\nabla^{\perp}_{\dt}H-\Delta^{\perp} H=\sum_{i=1}^m\pr\big(\rk(H,e_i,e_i)\big)+\sum_{i,j=1}^mA_H(e_i,e_j)A(e_i,e_j).
\end{equation*}
Combining the above two equalities, we obtain
\begin{eqnarray*}
\big(\nabla^{\perp}_{\partial_t}\vartheta\big)(\xi,\xi)&=&2\langle\Delta^{\perp} H,\xi\rangle H_{\xi} \\
&&-2\sum_{i=1}^{m}\rk\big(H,e_i,e_i,\xi\big)H_{\xi} \\
&&+2\sum_{i,j=1}^{m}A_{H}(e_i,e_j)A_{\xi}(e_i,e_j)H_{\xi}.
\end{eqnarray*}
The next step is to compute the Laplacian of the tensor $\vartheta$. At first let us compute the covariant
derivative. Fix a point $(x_0,t_0)$ in space time and let $\xi,\eta$ be two normal vector fields of $\Gamma(f_{t_0})$
defined in a neighborhood of $x_0$. Differentiating with respect to the direction $e_{i}$, we have
\begin{equation*}
\big(\nabla^{\perp}_{e_i}\vartheta\big)(\xi,\eta)=\big\langle\nabla^{\perp}_{e_i}H,\xi\big\rangle H_{\eta}
+\big\langle\nabla^{\perp}_{e_i}H,\eta\big\rangle H_{\xi}.
\end{equation*}
Differentiating once more with respect to the direction $e_i$ and summing up we deduce that
\begin{equation}
\big(\Delta^{\perp}\vartheta\big)(\xi,\xi)=2\big\langle\Delta^{\perp} H,\xi\big\rangle H_{\xi}
+2\sum_{i=1}^{m}\big\langle\nabla^{\perp}_{e_i}H,\xi\big\rangle^2.
\end{equation}
Combining the relations of the time derivative and of the Laplacian we obtain the desired evolution
equation. This completes the proof.
\end{proof}

\section{Proof of the theorem}
During this section we will always assume that $(M,\gm)$, $(N,\gn)$ and $f:M\to N$ satisfy the
assumption of the Theorem. The next lemma will be crucial to deal with the non-compactness of $N$.

\begin{lemma}
There exists a uniform positive constant $C$ such that
$$\|H\|^2(x,t)\le C,$$
for any $(x,t)\in M\times[0,T)$.
\end{lemma}
\begin{proof}
Consider the symmetric $2$-tensor $P$, defined on the normal bundles of the evolving graphs and given by
$$P:=\kappa\,\vartheta+\sind^{\perp},$$
where $\kappa$ is a sufficiently small positive constant such that $P<0$ at time $t=0$.
We claim now that, taking if necessary a smaller choice for $\kappa$,
the tensor $P$ remains negative definite in time. Assume in contrary that this is not true. Then, there will be a first
time such that $P$ admits a unit null-eigenvector
$\eta$ at a point $(x_0,t_0)$. Note that $\eta$ is normal at the graph at the point $(x_0,t_0)$.

According to the second derivative criterion \cite{hamilton2}, we have
\begin{enumerate}[$(a)$]
\item
\quad$P(\xi,\eta)=\kappa\,\vartheta(\xi,\eta)+\sind^{\perp}(\xi,\eta)=0$,\medskip
\item
\quad$(\nabla P)(\eta,\eta)=0$,\medskip
\item
\quad$(\nabla^{\perp}_{\dt}P-\Delta^{\perp} P\big)(\eta,\eta)\ge 0,$
\end{enumerate}
for any normal vector $\xi$ of the graph at the point $(x_0,t_0)$.

Estimating at $(x_0,t_0)$ we get from $(c)$ that
\begin{eqnarray}
0&\le&-\sum_{i,j,k=1}^mA_{\eta}(e_i,e_j)A_{\eta}(e_i,e_k)\sind(e_j,e_k)-\kappa\sum_{i=1}^m\langle\nabla^{\perp}_{e_i}H,\eta\rangle^{2}\nonumber\\
&&+\sum_{i,j=1}^m\Big\{A_{\eta}(e_i,e_j)\,\sk\big(A(e_i,e_j),\eta\big)+\kappa\,A_{H}(e_i,e_j)A_{\eta}(e_i,e_j)H_{\eta}\Big\}\nonumber\\
&&-\sum_{i,j=1}^m\Big\{\rk(e_i,e_j,e_i,\eta)\,\sk(e_j,\eta)+\kappa\rk\big(H,e_i,e_i,\eta\big)H_{\eta}\Big\}.\nonumber
\end{eqnarray}
Let $\{\xi_1,\dots,\xi_n\}$ be an orthonormal basis of $\mathcal{N}_{x_0}M$. Then,
\begin{eqnarray}
0&\le&-\sum_{i,j=1}^mA_{\eta}(e_i,e_j)A_{\eta}(e_i,e_k)\sind(e_j,e_k) \nonumber\\
&&+\sum_{l=1}^n\sum_{i,j=1}^m A_{\eta}(e_i,e_j)A_{\xi_l}(e_i,e_j){\sind^{\perp}}(\xi_l,\eta) \nonumber\\
&&+\sum_{l=1}^n\sum_{i,j=1}^m\kappa H_{\xi_l}H_{\eta}A_{\xi_l}(e_i,e_j)A_{\eta}(e_i,e_j)\nonumber\\
&&-\sum_{i,j=1}^m\Big\{\rk(e_i,e_j,e_i,\eta)\,\sk(e_j,\eta)+\kappa \rk\big(H,e_i,e_i,\eta\big)H_{\eta}\Big\}\nonumber.
\end{eqnarray}
Since,  for any $l\in\{1,\dots,n\}$, from $(a)$ it holds
$$\kappa H_{\xi_l}H_{\eta}=-\sind^{\perp}(\xi_l,\eta),$$
we finally get that
\begin{eqnarray}\label{est}
0&\le&-\sum_{i,j=1}^mA_{\eta}(e_i,e_j)A_{\eta}(e_i,e_k)\sind(e_j,e_k) \\
&&-\sum_{i,j=1}^m\Big\{\rk(e_i,e_j,e_i,\eta)\,\sk(e_j,\eta)+\kappa \rk\big(H,e_i,e_i,\eta\big)H_{\eta}\Big\}.\nonumber
\end{eqnarray}
Denote by $\mathcal{A}$ the first part of (\ref{est}) whose terms are involving the second fundamental form
and by $\mathcal{B}$ the remaining curvature terms. The idea is to show that $\mathcal{A}$ becomes sufficiently negative
for small choices of $\kappa$ and dominates $\mathcal{B}$ that depends only on the singular values and the geometry
of $M$ and $N$.

{\bf Fact 1}: Since $\sind^{\perp}$ remains negative in time, from Lemma \ref{length} it follows that there exists a
universal positive constant $\varepsilon$ such that
$$\varepsilon|\xi|^2\le-\sind^{\perp}(\xi,\xi)\le |\xi|^2,$$
for any $\xi$ in $\mathcal{N}M$. Note that at $(x_0,t_0)$ it holds,
\begin{equation}\label{eps}
\kappa H^2_{\eta}=\kappa\,\vartheta(\eta,\eta)=-\sind^{\perp}(\eta,\eta)\ge\varepsilon.
\end{equation}
Therefore, as $\kappa$ becomes smaller $H^2_{\eta}$ becomes larger.

{\bf Fact 2 }: From the relations (\ref{sind}) we deduce that
$$\mathcal{A}\le-\sum_{i,j=1}^mA^2_{\eta}(e_i,e_j)\sind(e_i,e_j)
\le-\varepsilon |A^2_{\eta}|\le-\frac{\varepsilon}{m}H^2_{\eta}.$$
Thus for sufficiently small values of $\kappa$, $\mathcal{A}$ becomes sufficiently negative.

{\bf Fact 3}: Note now that the first term of $\mathcal{B}$ depends only on the geometry of $(M,\gm)$ and
$(N,\gn)$ as well as on the singular values of $f_t$ which we know are bounded. The second term of
$\mathcal{B}$  also depends only on these data, since
\begin{eqnarray*}
\kappa\,\rk(H,e_i,e_i,\eta)H_{\eta}&=&\sum_{l=1}^n\kappa\,H_{\xi_{l}}H_{\eta}\rk(\xi_l,e_i,e_i,\eta)\\
&=&-\sum_{l=1}^n\sind^{\perp}(\xi_{l},\eta)\rk(\xi_l,e_i,e_i,\eta),
\end{eqnarray*}
where $\{\xi_1,\dots,\xi_n\}$ is a local basis on the normal bundle of the graph. Therefore, there exists a
universal constant $c:=c(M,N,\varepsilon)$ such that $\mathcal{B}\le c$. Therefore, due to relation (\ref{eps})
we get that
$$\mathcal{B}\le\frac{c}{\varepsilon}\varepsilon\le\frac{c}{\varepsilon}\kappa H^2_{\eta}.$$
Thus,
$$\mathcal{A}+\mathcal{B}\le\left(\frac{c}{\varepsilon}\kappa-\frac{\varepsilon}{m}\right)H^2_{\eta}$$
Consequently, for $\kappa<\varepsilon^2/c\,m$, we see that $\mathcal{A}+\mathcal{B}<0$ which contradicts
(\ref{est}). Therefore, the norm of the mean curvature vector field remains bounded in time. This completes
the proof of the lemma.
\end{proof}

\begin{remark}
As one can see from the proof, we make use only of the facts that $M$ is compact, $N$ is complete with
bounded sectional curvatures and that all the singular values of $f_t$ are bounded from above by a positive
universal constant which is less than $1$.
\end{remark}

The proof of the Theorem will be concluded by exploiting the blow up argument of Wang \cite{wang} and White's
regularity theorem \cite{white}. Let us recall at first the following crucial estimate.
\begin{lemma}[\cite{savas1}]\label{logu}
The following estimate holds,
$$\nabla_{\dt}\hspace{-2pt}\log {\det\big\{I+(\df)^{T}\df\big\}}\le\Delta\log {\det\big\{I+(\df)^{T}\df\big\}}
-\delta\|A\|^2,$$
for some positive real number $\delta$, where here $I$ stands for the unit matrix and $(\df)^T$ for the transpose of $\df$.
\end{lemma}
Once this estimate is available one can use White's regularity theorem \cite{white}
to exclude finite time singularities as long as on finite time intervals
the graphs stay in compact regions of $M\times N$, which clearly is true, if $M\times N$ is compact. In our case $N$ is complete but we may now exploit the mean curvature estimate of the Theorem to get the desired $C^0$-estimate
for the graphs on finite time intervals. To see this, fix a point $x\in M$ and consider the curve
$\gamma:[t_0,t_1]\to M\times N$, given by
$$\gamma(x,t):=F(x,t).$$
The length $L(\gamma)$ of $\gamma$ can be estimated using the bound of the mean curvature vector as follows
\begin{eqnarray*}
L(\gamma)&=&\int_{t_0}^{t_1}\left\|\frac{\dF}{\operatorname{d}{\hspace{-2pt}}t}(x,t)\right\|\operatorname{d}{\hspace{-2pt}}t\le\int_{t_0}^{t_1}\left\|H(x,t)\right\|\operatorname{d}{\hspace{-2pt}}t\le C(t_1-t_0)\\
&\le&CT,
\end{eqnarray*}
Therefore,
$$\operatorname{dist}\left(F(x,t_0),F(x,t_1)\right)\le L(\gamma)\le CT.$$
Suppose the graphs remain in a compact region
$W$ of $M\times N$ on a finite time interval $[0,T)$. By Nash's
embedding theorem \cite{nash}
one can embed $W$ isometrically in some euclidean space
$\real{p}$ and make sure that the isometric embedding has bounded geometry.
The bounded geometry is essential in the application of White's regularity theorem \cite{white} for the mean curvature flow with controlled error terms,
which by the compactness of $W$ is applied to the mean curvature flow of
$F(M)\subset W\subset\real{p}.$
Following the same arguments developed in the papers \cite[Section 4]{wang} or
\cite[Section 3]{lee}, we can prove the long-time existence and the convergence of
the mean curvature flow to a constant map.

% Literaturliste
%%%%%%%%%%%%%%%%%%%%%%%%%%%%%%%%%%%%%%%%%%%%%%%%%%%%%%%%%%%%%%%%%%%%%%%%%


\begin{bibdiv}
\begin{biblist}

\bib{chau1}{article}{
   author={Chau, A.},
   author={Chen, J.},
   author={He, W.},
   title={Lagrangian mean curvature flow for entire Lipschitz graphs},
   journal={Calc. Var. Partial Differential Equations},
   volume={44},
   date={2012},
   %number={1-2},
   pages={199--220},
   %issn={0944-2669},
   %review={\MR{2898776}},
   %doi={10.1007/s00526-011-0431-x},
}

\bib{guth1}{article}{
   author={Guth, L.},
   title={Contraction of areas vs. topology of mappings},
   journal={arXiv:1211.1057v3},
   %volume={88},
   date={2013},
   pages={1--81},
   %issn={0003-486X},
   %review={\MR{0233295 (38 \#1617)}},
}

\bib{guth}{article}{
   author={Guth, L.},
   title={Homotopy non-trivial maps with small $k$-dilation},
   journal={arXiv:0709.1241v1},
   %volume={88},
   date={2007},
   pages={1--7},
   %issn={0003-486X},
   %review={\MR{0233295 (38 \#1617)}},
}

\bib{hamilton2}{article}{
   author={Hamilton, R.},
   title={Three-manifolds with positive Ricci curvature},
   journal={J. Differential Geom.},
   volume={17},
   date={1982},
   %number={2},
   pages={255--306},
   %issn={0022-040X},
   %review={\MR{664497 (84a:53050)}},
}

\bib{lee}{article}{
   author={Lee, K.-W.},
   author={Lee, Y.-I.},
   title={Mean curvature flow of the graphs of maps between compact
   manifolds},
   journal={Trans. Amer. Math. Soc.},
   volume={363},
   date={2011},
   %number={11},
   pages={5745--5759},
   %issn={0002-9947},
   %review={\MR{2817407}},
   %doi={10.1090/S0002-9947-2011-05204-9},
}

\bib{nash}{article}{
   author={Nash, J.},
   title={The imbedding problem for Riemannian manifolds},
   journal={Ann. of Math. (2)},
   volume={63},
   date={1956},
   pages={20--63},
   %issn={0003-486X},
   %review={\MR{0075639 (17,782b)}},
}

\bib{savas1}{article}{
   author={Savas-Halilaj, A.},
   author={Smoczyk, K.},
   title={Homotopy of area decreasing maps by mean curvature flow},
   journal={arXiv:1302.0748},
   %volume={on line first},
   date={2013},
   pages={1--18},
   %issn={0003-486X},
   %review={\MR{0233295 (38 \#1617)}},
}


\bib{savas}{article}{
   author={Savas-Halilaj, A.},
   author={Smoczyk, K.},
   title={Bernstein theorems for length and area decreasing minimal maps},
   journal={Calc. Var. Partial Differ. Equ.},
   volume={on line first},
   date={2013},
   pages={1--29},
   %issn={0003-486X},
   %review={\MR{0233295 (38 \#1617)}},
}

\bib{smoczyk1}{article}{
   author={Smoczyk, K.},
   title={Mean curvature flow in higher codimension-Introduction and survey},
   journal={Global Differential Geometry,  Springer Proceedings in Mathematics},
   volume={12},
   date={2012},
   pages={231--274},
}

\bib{smoczyk}{article}{
   author={Smoczyk, K.},
   title={Long-time existence of the Lagrangian mean curvature flow},
   journal={Calc. Var. Partial Differential Equations},
   volume={20},
   date={2004},
   %number={1},
   pages={25--46},
   %issn={0944-2669},
   %review={\MR{2047144 (2004m:53119)}},
   %doi={10.1007/s00526-003-0226-9},
}

\bib{tsui}{article}{
   author={Tsui, M.-P.},
   author={Wang, M.-T.},
   title={Mean curvature flows and isotopy of maps between spheres},
   journal={Comm. Pure Appl. Math.},
   volume={57},
   date={2004},
   %number={8},
   pages={1110--1126},
   %issn={0010-3640},
   %review={\MR{2053760 (2005b:53110)}},
   %doi={10.1002/cpa.20022},
}

\bib{wang}{article}{
   author={Wang, M.-T.},
   title={Long-time existence and convergence of graphic mean curvature flow
   in arbitrary codimension},
   journal={Invent. Math.},
   volume={148},
   date={2002},
   %number={3},
   pages={525--543},
   %issn={0020-9910},
   %review={\MR{1908059 (2003b:53073)}},
   %doi={10.1007/s002220100201},
}

\bib{white}{article}{
   author={White, B.},
   title={A local regularity theorem for mean curvature flow},
   journal={Ann. of Math. (2)},
   volume={161},
   date={2005},
   %number={3},
   pages={1487--1519},
   %issn={0003-486X},
   %review={\MR{2180405 (2006i:53100)}},
   %doi={10.4007/annals.2005.161.1487},
}


\end{biblist}
\end{bibdiv}
\end{document}